\documentclass[11pt]{article}
\usepackage{amsmath,amssymb,latexsym,amsfonts,color,natbib,dsfont,url,breakcites}
\usepackage[breaklinks=true]{hyperref}
\allowdisplaybreaks
\title{Spectral tail processes and max-stable approximations of multivariate 
regularly varying time series}
\author{
Anja Jan{\ss}en\footnote{University of Copenhagen, Department of Mathematics, 
Universitetsparken 5, 2100 Copenhagen, Denmark;
email: anja@math.ku.dk}}

\newcommand{\R}{\mathbb{R}}
\newcommand{\N}{\mathbb{N}}

\newcommand{\Z}{\mathbb{Z}}

\usepackage[thmmarks,amsmath,amsthm]{ntheorem}

\theorempreskipamount 2ex \theorempostskipamount 3ex \theoremsymbol{$\Box$}
\theoremseparator{\quad}

\theorembodyfont{\em}

\newtheorem{theorem}{Theorem}[section]
\newtheorem{corollary}[theorem]{Corollary}
\newtheorem{remark}[theorem]{Remark}
\newtheorem{proposition}[theorem]{Proposition}
\newtheorem{lemma}[theorem]{Lemma}
\newtheorem*{property}[theorem]{Property}

\newtheorem{definition}[theorem]{Definition}
\numberwithin{equation}{section}
\theorembodyfont{\rm}
\newtheorem{example}[theorem]{Example}

\def\Min(#1,#2){#1\wedge #2}
\def\Max(#1,#2){#1\vee #2}

\begin{document}

\maketitle

\begin{abstract}
A regularly varying time series as introduced in \cite{BaSe09} is a 
(multivariate) time series such that all finite dimensional distributions are 
multivariate regularly varying. The extremal behavior of such a process can 
then be described by the index of regular variation and the so-called spectral 
tail process, which is the limiting distribution of the rescaled process, given 
an extreme event at time 0. As shown in \cite{BaSe09}, the stationarity of the 
underlying time series implies a certain structure of the spectral tail 
process, informally known as the ``time change formula''. In this article, we 
show that on the other hand, every process which satisfies this property is in 
fact the spectral tail process of an underlying stationary max-stable process. 
The spectral tail process and the corresponding max-stable process then provide 
two complementary views on the extremal behavior of a multivariate regularly 
varying stationary time series. 
\end{abstract}

\section{Introduction}\label{Sec:Intro}
The concept of regular variation has become a standard tool for the extremal 
analysis of multivariate time series. Roughly speaking, the assumption of 
regular variation of a time series implies that extremal episodes from this 
time series can be modelled as the product of a heavy-tailed radial component 
which determines the magnitude of the extremal event and an independent and 
normalized random vector which determines the dependence structure between the 
different components of one observation and over time. There are different ways 
of defining this second random component, depending on the specification of an 
``extremal event''. In the case of a stationary time series, one often looks at 
the limiting distribution of the rescaled process conditioned on the event that 
the observation at time 0 exceeds in norm a threshold which tends to infinity. 

In the following, let $\| \cdot\|$ be an arbitrary but fixed norm on 
$\mathbb{R}^d$. Write $\mathcal{L}(X)$ for the distribution of a random 
quantity $X$ and $\mathcal{L}(X|A)$ or $\mathcal{L}(X|Y)$ for the distribution 
of $X$ conditioned on the event $A$ or the $\sigma$-algebra which is generated 
from the random quantity $Y$. We denote weak convergence by 
$\overset{w}{\Rightarrow}$. 
\newpage
\begin{definition}\label{rem:MRTS}
Let $(X_t)_{t \in \mathbb{Z}}$ be a stationary time series with values in $ 
\mathbb{R}^d$. If there exists a time series $(\Theta_t)_{t \in \mathbb{Z}}$, 
such that
\begin{equation}\label{convtotheta} \mathcal{L}\left( 
\left(\frac{X_{s}}{\|X_0\|}, \ldots, \frac{X_t}{\|X_0\|}\right)\Bigg| \, 
\|X_0\|>x\right) \overset{w}{\Rightarrow} \mathcal{L}(\Theta_s, \ldots, 
\Theta_t), \;\;\; x \to \infty,
\end{equation}
for all $s,t \in \mathbb{Z}$, then we call $(\Theta_t)_{t \in \mathbb{Z}}$ the 
\emph{spectral tail process} of $(X_t)_{t \in \mathbb{Z}}$. The process 
$(\Theta_t)_{t \in \mathbb{N}_0}$ is called the \emph{forward spectral tail 
process} and $(\Theta_{-t})_{t \in \mathbb{N}_0}$ is called the \emph{backward 
spectral tail process}.

If there exists a time series $(Y_t)_{t \in \mathbb{Z}}$ such that
\begin{equation}\label{convtotail} \mathcal{L}\left( \left(\frac{X_{s}}{x}, 
\ldots, \frac{X_t}{x}\right)\Bigg| \, \|X_0\|>x\right) \overset{w}{\Rightarrow} 
\mathcal{L}(Y_s, \ldots, Y_t), \;\;\; x \to \infty,
\end{equation}
for all $s,t \in \mathbb{Z}$, then we call $(Y_t)_{t \in \mathbb{Z}}$ the 
\emph{tail process} of $(X_t)_{t \in \mathbb{Z}}$.
\end{definition}
\begin{remark}\label{rem:tailproc}
As shown in \cite{BaSe09}, Theorems 2.1 and 3.1, for a stationary time series 
$(X_t)_{t \in \mathbb{Z}}$ the two convergences \eqref{convtotheta} and 
\eqref{convtotail} are equivalent and we have
$$ (Y_t)_{t \in \Z}\overset{d}{=} (Y \cdot \Theta_t)_{t \in \Z},$$
for a Pareto$(\alpha)$-distributed r.v.\ $Y$ independent of $(\Theta_t)_{t \in 
\Z}$ (where $\overset{d}{=}$ denotes equality in distribution). The value of 
$\alpha$ is called the \emph{index of regular variation} of $(X_t)_{t \in \Z}$. 
Furthermore, any of these two convergences is equivalent to that the time 
series is regularly varying, i.e.\ all its finite dimensional distributions are 
multivariate regularly varying. 
\end{remark}
In the following we will always deal with a stationary underlying process 
$(X_t)_{t \in \Z}$. It is important to note that the spectral tail process 
$(\Theta_t)_{t \in \Z}$ of an underlying stationary process is in general not a 
stationary time series, since we condition on a particular event that happens 
at time 0. However, one can show that the resulting spectral tail process of a 
stationary underlying time series satisfies a different property instead. In 
the following we write $0$ for the real number or a vector or sequence 
consisting of all zeros. The meaning should be clear from the context.
\begin{property}[TCF]\label{A1}
We say that a time series $(\Theta_t)_{t \in \mathbb{Z}}$ with values in 
$\mathbb{R}^d$ satisfies \emph{Property (TCF)} if $P(\|\Theta_0\|=1)=1$ and the 
so-called ``time-change formula'' (\cite{SeZhMe17}) holds, i.e.\ if for 
$\alpha>0$
 \begin{equation}\label{TCF} E(f(\Theta_{s-i}, \ldots, 
\Theta_{t-i}))=E\left(f\left(\frac{\Theta_s}{\|\Theta_i\|}, \ldots, 
\frac{\Theta_t}{\|\Theta_i\|}\right)\mathds{1}_{\{\|\Theta_i\|>0\}}
\|\Theta_i\|^\alpha\right)
 \end{equation}
for all $s \leq 0 \leq t, i \in \mathbb{Z},$ and all bounded and continuous 
functions $f:(\mathbb{R}^d)^{t-s+1} \to \mathbb{R}$ such that $f(\theta_s, 
\ldots, \theta_t)=0$ whenever $\theta_0=0$. 
\end{property}
Note that the above property depends on the parameter $\alpha>0$, so it would 
be more precise to speak of Property (TCF($\alpha$)) instead, but for reasons 
of brevity we omit the parameter which we assume to be fixed throughout. 
\begin{remark}\label{rem:MRTSrep}
\begin{enumerate}
 \item We have added the indicator function in \eqref{TCF} in order to make 
clear that the expression in the expected value is to be interpreted as $0$ on 
the set where $\|\Theta_i\|=0$ (and thus the argument of the function $f$ is 
not defined). In the following we will in most cases omit the indicator 
function 
for the sake of brevity.
 \item Similar as in \cite{JaSe14} one can show that Property (TCF) implies 
that \eqref{TCF} even holds for all bounded and {\em measurable} functions $f$ 
such that $f(\theta_s, \ldots, \theta_t)$ $=0$ whenever $\theta_0=0$. 
 \item As the law of $(\Theta_t)_{t \in \mathbb{Z}}$ is completely determined by its finite-dimensional distributions,
Property (TCF) implies that 
 \begin{equation}\label{Eq:TCFfull} E(f((\Theta_{t-i})_{t \in 
\mathbb{Z}}))=E\left(f\left(\left(\frac{\Theta_t}{\|\Theta_i\|}\right)_{t \in 
\mathbb{Z}}\right)\|\Theta_i\|^\alpha\right), \;\;\; i \in \Z,
 \end{equation}
 for all bounded functions $f:(\mathbb{R}^d)^{\mathbb{Z}} \to \mathbb{R}$ such 
that $f((\theta_t)_{t \in \mathbb{Z}})=0$ whenever $\theta_0=0$ and $f$ is 
$\mathcal{B}((\mathbb{R}^d)^{\mathbb{Z}})$-$\mathcal{B}(\R)$-measurable, where 
$\mathcal{B}(\cdot)$ stands for the corresponding Borel $\sigma$-algebra.
 \end{enumerate}
\end{remark}
As shown in \cite{BaSe09}, Theorem 3.1, Property (TCF) (with $\alpha$ being 
equal to the index of regular variation) is always satisfied for a spectral 
tail process of an underlying stationary regularly varying process. See also 
\cite{SeZhMe17} for the more general case of $X_t$ taking values in a 
``star-shaped'' metric space. 

We will subsequently deal with the question whether in turn Property (TCF) implies 
that the corresponding process is a spectral tail process of some underlying 
process. We show that the answer is yes and we will construct corresponding 
underlying max-stable processes, see Theorem \ref{the:general}. It will turn 
out that the case where additionally a summability or short-range dependence 
condition (see Section \ref{Sec:ShiftInv} for details) is satisfied allows for 
a particularly simple construction which is based on multivariate mixed moving 
maxima processes, see Theorem \ref{the:main} and Corollary \ref{Cor:main}. This 
construction allows us to connect the extremal behavior of the process 
conditioned on the specific event of an extremal exceedance at time 0 with the 
overall extremal behavior over time as modelled by a max-stable process, see 
Proposition \ref{Prop:max-stable}. The spectral tail process and the 
approximation of extremal events by a max-stable process thus provide two 
complementary views on the extremal behavior of a regularly varying time 
series. Additionally, this point of view also gives rise to a theoretical 
motivation (Proposition \ref{prop:cond}) for the POT method for dependent 
observations which uses clusterwise maxima for estimation of extremal 
parameters, see \cite{DaSm90}.

The very recent works \cite{PlSo17} and \cite{DoHaSo17} address similar questions, where properties of the spectral tail process and max-stable representations are analyzed with the help of the so-called tail measure, which is an extension of the limit measure in multivariate regular variation to the sequence space. Our approach here is more focussed on the distribution of the spectral tail process itself. It allows for a new interpretation of the property (TCF), see Theorem \ref{the:RS}, and representations of corresponding underlying max-stables processes which are generated from i.i.d.\ copies of the spectral tail process, thereby facilitating for example simulation from these processes.

There are also some links between our work and \cite{EnMaOeSc14}, who explore 
connections between different representations of univariate max-stable 
processes. However, our approach differs substantially by starting solely from Property (TCF) and focussing on the particular properties implied by 
stationarity of processes. Finally, some connections exist between the topics 
studied here and theoretical properties of stationary max- and sum-stable 
processes which were analyzed in (among others) \cite{Ro95}, \cite{RoSa08}, 
\cite{WaSt10} and \cite{DoKa16}. But there the analysis starts from a stable 
process and is focussed on suitable decompositions of this process, while our 
approach is somehow the other way round by starting with the spectral process 
and finding a suitable composition of copies from this process in order to 
generate a max-stable process with given properties. Furthermore, the 
aforementioned works rely heavily on tools from dynamical systems theory, while our approach is solely based on the description of finite-dimensional 
distributions of stochastic processes and point process techniques, which are 
two standard tools of extreme value theory for time series. 

The rest of the paper is organized as follows: Section \ref{Sec:ShiftInv} 
introduces and analyzes a short-range dependence condition which, if satisfied 
in addition to Property (TCF), allows us to construct an underlying mixed 
moving maxima process with given spectral tail process, as shown in Section 
\ref{Sec:Construction}. At the end of this section, we also discuss 
implications of our results for statistical analysis. Section \ref{Sec:general} 
finally treats the general case where we show that again an underlying 
stationary process exists but has a more complicated representation than in the 
short-range dependence setting.

\section{A shift invariance property derived from Property 
(TCF)}\label{Sec:ShiftInv}
We start by showing that Property (TCF), together with a summability 
assumption, implies a certain shift invariance of the process $(\Theta_t)_{t 
\in \mathbb{Z}}$. 
 \begin{property}[SC] We say that a time series $(\Theta_t)_{t \in \mathbb{Z}}$ 
with values in $\mathbb{R}^d$ satisfies the summability condition $\mbox{(SC)}$ 
for $\alpha>0$ if 
$$0<\sum_{t \in \mathbb{Z}} \| \Theta_t \|^\alpha < \infty\;\;\;  \mbox{a.s.} $$
 \end{property}
 
 The same comment as for Property (TCF) about the omission of the parameter 
$\alpha$ also applies here. 
 In case that Property (SC) is satisfied, we introduce the following notation
 $$ \|\Theta\|_\alpha=\| (\Theta_t)_{t \in \mathbb{Z}}\|_\alpha :=\left(\sum_{t 
\in \mathbb{Z}}\| \Theta_t\|^\alpha \right)^{1/\alpha}.$$
 For $\alpha \geq 1$, this is just the $L_\alpha$-norm of $(\|\Theta_t\|)_{t 
\in \mathbb{Z}}$, but remember that $\alpha<1$ is also possible. 

For further analysis of Property (SC) we also introduce
\begin{equation}\label{Eq:Thetastar} \| \Theta^\ast 
\|=\|\Theta^\ast\|((\Theta_t)_{t \in \mathbb{Z}}):=\sup_{t \in \mathbb{Z}}\| 
\Theta_t\| \in [0, \infty ] 
\end{equation}
and 
\begin{equation}\label{Eq:Tstar} T^\ast =T^\ast((\Theta_t)_{t \in 
	\mathbb{Z}}):= \inf \{t \in \mathbb{Z}: \| \Theta_t\|=\|\Theta^\ast \| \} \in 
\overline{\mathbb{Z}}:=\mathbb{Z} \cup \{-\infty\} \cup \{\infty\}. 
\end{equation}
As usual, we set $\inf(\mathbb{Z})=-\infty$ and $\inf(\emptyset)=\infty$. So, 
in particular, $\| \Theta^\ast \|=\infty$ implies $T^\ast=\infty$. 
 
At first, Property (SC) may look rather restrictive, but the following lemma shows that once Property (TCF) is satisfied, we have equivalent assumptions which seem reasonable for a large class of processes, see also Remark \ref{rem:sum}. See also \cite{PlSo17}, Corollary 3.3, for this statement and an alternative proof. 

 \begin{lemma}\label{equivalence}
 Let Property (TCF) hold. Then the following statements are equivalent:
 \begin{enumerate}
  \item $\sum_{t \in \mathbb{Z}} \| \Theta_t \|^\alpha < \infty \;\;$ a.s. 
(i.e. Property (SC) holds)
  \item $\|\Theta_t\| \to 0 $ a.s.\ for $|t| \to \infty$.
  \item $P(T^\ast \in \mathbb{Z})=1$
 \end{enumerate}
 \end{lemma}
 
 \begin{proof}
 If (i) holds, then necessarily $\| \Theta_t \| \to 0$ a.s.\ as $|t| \to 
\infty$ and thereby (ii) follows. Since furthermore $\| \Theta_0\|=1$ a.s., statement (ii) implies that both 
$\|\Theta^\ast\|$ and $T^\ast$ are finite a.s. and (iii) holds. It is left to show that (iii) implies (i).
 
Let (iii) hold and assume that $P(\sum_{k \in \mathbb{Z}} \| \Theta_k \|^\alpha= 
\infty)>0$. Then there exists an $i^\ast \in \Z$ such that $P(\sum_{k \in 
\mathbb{Z}} \| \Theta_k \|^\alpha=\infty, T^\ast=i^\ast)>0$ and therefore
\begin{eqnarray*} \infty&=&E\left(\mathds{1}_{\{T^\ast=i^\ast\}}\sum_{k \in 
\mathbb{Z}} \| \Theta_k \|^\alpha\right) \\
&=&\sum_{k \in \mathbb{Z}}  E\left(\mathds{1}_{\{T^\ast=i^\ast\}} \| \Theta_k 
\|^\alpha\right) \\
&=& \sum_{k \in \mathbb{Z}}  
E\left(\mathds{1}_{\{\|\Theta_j\|<\|\Theta_{i^\ast}\|, j<i^\ast, 
\|\Theta_j\|\leq \|\Theta_{i^\ast}\|, j\geq i^\ast\}} 
\mathds{1}_{\{\Theta_0\neq 0\}}\| \Theta_k \|^\alpha\right) \\
&=& \sum_{k \in \mathbb{Z}}  
E\left(f\left(\left(\frac{\Theta_t}{\|\Theta_k\|}\right)_{t \in 
\Z}\right)\|\Theta_k \|^\alpha\right) 
\end{eqnarray*}
with
$$ f((\theta_t)_{t \in \Z})= \mathds{1}_{\{\|\theta_j\|<\|\theta_{i^\ast}\|, 
j<i^\ast, \|\theta_j\|\leq \|\theta_{i^\ast}\|, j\geq 
i^\ast\}}\mathds{1}_{\{\theta_0\neq 0\}}, $$
which is bounded, measurable and $f((\theta_t)_{t \in \mathbb{Z}})=0$ whenever 
$\theta_0=0$. Therefore, with \eqref{Eq:TCFfull} applied to each summand, this 
expression equals
\begin{eqnarray*}
&& \sum_{k \in \mathbb{Z}}  E\left(f\left(\left(\Theta_{t-k}\right)_{t \in 
\Z}\right)\right)\\
&=&  \sum_{k \in \mathbb{Z}}  
E\left(\mathds{1}_{\{\|\Theta_j\|<\|\Theta_{i^\ast-k}\|, j<i^\ast-k, 
\|\Theta_j\|\leq \|\Theta_{i^\ast-k}\|, j\geq i^\ast-k\}} 
\mathds{1}_{\{\Theta_{-k}\neq 0\}}\right) \\
&\leq &  \sum_{k \in \mathbb{Z}}  
E\left(\mathds{1}_{\{\|\Theta_j\|<\|\Theta_{i^\ast-k}\|, j<i^\ast-k, 
\|\Theta_j\|\leq \|\Theta_{i^\ast-k}\|, j\geq i^\ast-k\}}\right)\\
&=& \sum_{k \in \mathbb{Z}} P(T^\ast=i^\ast-k)=1,
\end{eqnarray*}
which leads to a contradiction to our assumption that $P(\sum_{k \in 
\mathbb{Z}} \| \Theta_k \|^\alpha= \infty)>0$ and thereby proves the statement.
 \end{proof}
\begin{remark}\label{rem:sum}
 The assumption $\|\Theta_t\|\to 0$ as $|t|\to 
\infty$ excludes a sort of long-range dependence in extremes and it 
has been shown in \cite{BaSe09}, Proposition 4.2, that it is implied by a 
property introduced in \cite{DaHs95}: For a sequence $a_n \to \infty$ with 
$nP(\|X_0\|>a_n) \to c>0$, it is satisfied if there exists a 
sequence $r_n\to \infty, r_n/n \to 0$ such that
 \begin{equation}\label{Eq:finitemeancluster} \lim_{m \to \infty}\limsup_{n \to \infty}P\left(\max_{m \leq |t|\leq 
r_n}\|X_t\|>a_nu \, \middle| \, \|X_0\|>a_n u\right)=0 \;\;\; \mbox{ for all }u \in (0,\infty).
\end{equation}
 This property is sometimes called ``anti-clustering condition'' or 
``finite mean cluster size condition'' and is frequently used in the 
literature, 
c.f., e.g., \cite{DaMi98}, \cite{BaKrSe12} or \cite{MiWi16}. It is satisfied for a large variety of time series models such as ARMA models, Max-moving average processes, stochastic volatility models or GARCH(p,q) processes (under mild assumptions about coefficients and innovations, respectively), cf.\ \cite{BaDaMi02} and \cite{MiZh14}. 

So, by Lemma \ref{equivalence}, if $(\Theta_t)_{t \in \mathbb{Z}}$ is a spectral tail process and the underlying process satisfies \eqref{Eq:finitemeancluster}, then $(\Theta_t)_{t \in \mathbb{Z}}$ satisfies Property (SC). 

Alternatively, a Markovian structure of the spectral tail process as discussed in \cite{JaSe14} can simplify the task to check whether $\| \Theta_t \| \to 0$ a.s. as $|t| \to 
\infty$ and thus Property (SC) is satisfied.
\end{remark}

Under the assumption of Property (SC) we will now formulate an equivalent 
statement for Property (TCF). This equivalence allows for a probabilistic 
interpretation of the time change formula in form of an invariance property of 
$(\Theta_t)_{t \in \Z}$ under a specific random shift of time. 
\begin{theorem}\label{the:RS}
Let $(\Theta_t)_{t \in \Z}$ with values in $\mathbb{R}^d$ be a time series 
which satisfies Property (SC). Furthermore, let $(\Theta_t^{\mbox{\scriptsize 
RS}})_{t \in \Z}$ be a time series such that
\begin{equation}\label{Eq:RSdef1}(\Theta_t^{\mbox{\scriptsize RS}})_{t \in 
\Z}\overset{d}{=}\left(\frac{\Theta_{t+K(\Theta)}}{\|\Theta_{K(\Theta)}\|}
\right)_{t \in \Z}, 
\end{equation}
where $K(\Theta)=K((\Theta_t)_{t \in \Z})$ is a random integer with conditional 
probability mass function
\begin{equation}\label{Eq:RSdef2} P(K(\Theta)=k | (\Theta_t)_{t \in 
\Z})=\frac{\|\Theta_k\|^\alpha}{\sum_{t \in 
\Z}\|\Theta_t\|^\alpha}=\frac{\|\Theta_k\|^\alpha}{\|\Theta\|_\alpha^\alpha}, 
\;\;\; k \in \Z. 
\end{equation}
Then
\begin{equation}\label{Eq:RSsame} (\Theta_t^{\mbox{\scriptsize RS}})_{t \in 
\Z}\overset{d}{=}(\Theta_t)_{t \in \Z} 
\end{equation}

if and only if $(\Theta_t)_{t \in \Z}$ satisfies Property (TCF). 
\end{theorem}
The index ${\mbox{\it \scriptsize RS}}$ in the time series above may be 
interpreted as standing both for ``random shift" and ``re-scaled". 
\begin{proof} 
$1. ``\Rightarrow"$: Assume that all assumptions of Theorem \ref{the:RS} are 
met and that \eqref{Eq:RSsame} holds. Since $\|\Theta_0^{\mbox{\scriptsize 
RS}}\|=1$ a.s.\ by \eqref{Eq:RSdef1}, we have $P(\|\Theta_0\|=1)=1$, i.e.\ the 
first part of Property (TCF) is satisfied. For $s\leq 0 \leq t$ let 
$f:(\mathbb{R}^d)^{t-s+1} \to \mathbb{R}$ be a bounded and continuous function 
such that $f(y_s, \ldots, y_t)=0$ whenever $y_0=0$. Then, for $i \in \Z$,
\begin{eqnarray*}
&& E(f(\Theta_{s-i}, \ldots, \Theta_{t-i}))\\
&=& E\left(\sum_{k \in 
\Z}\frac{\|\Theta_k\|^\alpha}{\|\Theta\|_\alpha^\alpha}f(\Theta_{s-i}, \ldots, 
\Theta_{t-i})\right) \\
&=& E\left(\sum_{k \in \Z}\frac{\|\Theta^{\mbox{\scriptsize 
RS}}_k\|^\alpha}{\|\Theta^{\mbox{\scriptsize 
RS}}\|_\alpha^\alpha}f(\Theta^{\mbox{\scriptsize RS}}_{s-i}, \ldots, 
\Theta^{\mbox{\scriptsize RS}}_{t-i})\right) \\
&=& E\left(E\left(\sum_{k \in 
\Z}\frac{\|\Theta_{k+K(\Theta)}/\|\Theta_{K(\Theta)}\|\|^\alpha}{
\|\Theta/\|\Theta_{K(\Theta)}\|\|_\alpha^\alpha}f\left(\frac{\Theta_{
s-i+K(\Theta)}}{\|\Theta_{K(\Theta)}\|}, \ldots, 
\frac{\Theta_{t-i+K(\Theta)}}{\|\Theta_{K(\Theta)}\|}\right)\middle| 
(\Theta_t)_{t \in \Z}\right)\right)\\
&=& E\left(\sum_{k\in\Z} \sum_{l \in 
\Z}\frac{\|\Theta_l\|^\alpha}{\|\Theta\|_\alpha^\alpha}\frac{\|\Theta_{k+l}
\|^\alpha}{\|\Theta\|_\alpha^\alpha}f\left(\frac{\Theta_{s-i+l}}{\|\Theta_l\|}, 
\ldots, \frac{\Theta_{t-i+l}}{\|\Theta_l\|}\right)\right),
\end{eqnarray*}
where we used \eqref{Eq:RSdef1}, \eqref{Eq:RSdef2} and the assumption 
\eqref{Eq:RSsame}. The last expression simplifies to
\begin{eqnarray*}
&& E\left(\sum_{l\in\Z} 
\frac{\|\Theta_l\|^\alpha}{\|\Theta\|_\alpha^\alpha}f\left(\frac{\Theta_{s-i+l}}
{\|\Theta_l\|}, \ldots, \frac{\Theta_{t-i+l}}{\|\Theta_l\|}\right)\right)\\
&=& E\left(\sum_{m \in \Z} 
\frac{\|\Theta_{i+m}\|^\alpha}{\|\Theta\|_\alpha^\alpha}f\left(\frac{\Theta_{s+m
}}{\|\Theta_{i+m}\|}, \ldots, 
\frac{\Theta_{t+m}}{\|\Theta_{i+m}\|}\right)\right), \;\;\; \mbox{ with } 
m=l-i,\\
&=& E\left(\sum_{m \in \Z} 
\frac{\|\Theta_{m}\|^\alpha}{\|\Theta\|_\alpha^\alpha}f\left(\frac{\Theta_{s+m}
/\|\Theta_m\|}{\|\Theta_{i+m}/\|\Theta_m\|\|}, \ldots, 
\frac{\Theta_{t+m}/\|\Theta_m\|}{\|\Theta_{i+m}/\|\Theta_m\|\|}\right)\frac{
\|\Theta_{i+m}\|^\alpha}{\|\Theta_m\|^\alpha}\right)\\
&=& E\left(f\left(\frac{\Theta_s}{\|\Theta_i\|}, \ldots, 
\frac{\Theta_t}{\|\Theta_i\|}\right)\|\Theta_i\|^\alpha\right),
\end{eqnarray*}
where in the last step again \eqref{Eq:RSsame} was used. In the penultimate 
step we used the assumption about $f$ which guarantees that the corresponding 
summand in the expected value vanishes as soon as $\|\Theta_m\|=0$, which 
allows 
to expand the fraction by $\|\Theta_m\|^\alpha$ (again, all expressions are 
assumed to equal 0 as soon as one of the factors is 0, cf.\ Remark 
\ref{rem:MRTSrep} (i)). This proves Property (TCF). 

$2. ``\Leftarrow"$: Assume now that Property (TCF) holds in addition to 
Property (SC) and let $(\Theta_t^{\mbox{\scriptsize RS}})_{t \in \Z}$ have a 
distribution as defined in \eqref{Eq:RSdef1} and \eqref{Eq:RSdef2}. Then it is 
sufficient to show for all bounded and measurable functions 
$f:(\mathbb{R}^d)^{\mathbb{Z}} \to \mathbb{R}$ such that $f((\theta_t)_{t \in 
\mathbb{Z}})=0$ whenever $\theta_0=0$, that
$$ E(f((\Theta_t)_{t \in \Z}))=E(f((\Theta_t^{\mbox{\scriptsize RS}})_{t \in 
\Z})).$$
(The above equality has to hold for all bounded and measurable functions in 
order to have equality in distribution for both processes, but by switching 
from $f$ to $f'((\theta_t)_{t \in \mathbb{Z}})\:=f((\theta_t)_{t \in 
\mathbb{Z}})\cdot \mathds{1}_{\{\theta_0\neq 0\}}$ the above equality is 
sufficient since $P(\Theta_0=0)=P(\Theta_0^{\mbox{\scriptsize RS}}=0)=0$ by 
Property (TCF) and the construction of $(\Theta_t^{\mbox{\scriptsize RS}})_{t 
\in \Z}$.) Set
$$ \tilde{f}((\theta_t)_{t \in 
\mathbb{Z}}):=f\left(\left(\frac{\theta_t}{\|\theta_0\|}\right)_{t \in 
\mathbb{Z}}\right)\frac{\|\theta_0\|^\alpha}{\|(\theta_t)_{t \in 
\Z}\|_\alpha^\alpha}\mathds{1}_{\{\theta_0\neq 0\}}, $$
such that, by using the boundedness of $f$ in order to interchange sum and 
expectation,
\begin{eqnarray*}
 E(f((\Theta_t^{\mbox{\scriptsize RS}})_{t \in \Z}))&=& E\left(\sum_{k \in 
\Z}\frac{\|\Theta_{-k}\|^\alpha}{\|\Theta\|_\alpha^\alpha}f\left(\left(\frac{
\Theta_{t-k}}{\|\Theta_{-k}\|}\right)_{t \in \Z}\right)\right)\\
 &=& \sum_{k \in \Z}E\left(\tilde{f}((\Theta_{t-k})_{t \in \Z})\right) \\
 &=& \sum_{k \in 
\Z}E\left(\tilde{f}\left(\left(\frac{\Theta_t}{\|\Theta_k\|}\right)_{t \in 
\mathbb{Z}}\right)\|\Theta_k\|^\alpha\right) \\
 &=& \sum_{k \in 
\Z}E\left(f\left(\left(\frac{\Theta_t/\|\Theta_k\|}{\|\Theta_0/\|\Theta_k\|\|}
\right)_{t \in 
\mathbb{Z}}\right)\frac{\|\Theta_0/\|\Theta_k\|\|^\alpha}{
\|\Theta/\|\Theta_k\|\|_\alpha^\alpha}\|\Theta_k\|^\alpha\right)\\
 &=& \sum_{k \in \Z}E\left(f((\Theta_t)_{t \in 
\mathbb{Z}})\frac{\|\Theta_k\|^\alpha}{\|\Theta\|_\alpha^\alpha}
\right)=E\left(f((\Theta_t)_{t \in \mathbb{Z}}\right)
\end{eqnarray*}
where we have used \eqref{Eq:TCFfull}. This concludes the proof.
\end{proof}

\section{Construction of a max-stable process with given spectral tail process 
under Property (SC)}\label{Sec:Construction}
In this section, we will show that for each process $(\Theta_t)_{t \in \Z}$ 
which satisfies Properties (TCF) and (SC) there exists an underlying max-stable 
process  which has $(\Theta_t)_{t \in \Z}$ as corresponding spectral tail 
process. The general case where Property (SC) does not need to hold will be 
treated later in Section \ref{Sec:general}. For simplicity, we restrict 
ourselves to the case that all components of $\Theta_t$ are non-negative, but 
see Remark \ref{rem:nonpos} for the general case. Since the term ``max-stable'' 
has different interpretations in the context of multivariate stochastic 
processes in the literature, we start with a definition.
\begin{definition}\label{def:maxstable}
We call a stochastic process $(X_t)_{t \in \mathbb{Z}}$ with values in 
$[0,\infty)^d$ \emph{max-stable with index $\alpha$}, if for all $k \in \N, 
s\leq t, x_s, \ldots, x_t \in [0,\infty]^d$
\begin{equation}\label{eq:defms}(P(X_s\leq x_s, \ldots, X_t\leq 
x_t))^k=P(X_s\leq k^{-1/\alpha}x_s, \ldots, X_t\leq k^{-1/\alpha}x_t),
\end{equation}
where all inequalities are to be interpreted componentwise.
\end{definition}
The next theorem has some similarities to Theorem 4.1 in \cite{EnMaOeSc14}. 
Compared to their result we restrict ourselves to the integers as the index set 
but as an extension we allow both for multivariate observations and also for 
$\Theta_t=0$ for some $t \in \Z$. 
\begin{theorem}\label{the:main}
Let $(\Theta_t)_{t \in \Z}$ with values in $[0,\infty)^d$ be a stochastic 
process which satisfies Property (SC). Furthermore, let $(U_i, 
T_i,(\Theta_t^{(i)})_{t \in \mathbb{Z}})_{i \in \mathbb{N}}$ be an enumeration 
of points from a Poisson point process on $(0,\infty) \times \mathbb{Z} \times  
([0,\infty)^d)^{\Z}$ with intensity $\alpha u^{-\alpha-1} du \otimes 
\lambda(dt) \otimes P^{(\Theta_t)_{t \in \Z}}(d\theta)$, where $\lambda$ 
denotes 
the counting measure on $\mathbb{Z}$, i.e. $\lambda(B)=|B|$ for $B \subset 
\mathbb{Z}$. 
Then the stochastic process
\begin{equation}\label{Eq:max-stable} (Z_t)_{t \in \Z}=\left(\bigvee_{i \in \N} 
U_{i} \frac{\Theta^{(i)}_{t+T_i}}{\|(\Theta_l^{(i)})_{l \in 
\mathbb{Z}}\|_\alpha}  \right)_{t \in \mathbb{Z}} \end{equation}
is an almost surely finite, stationary and max-stable process with index 
$\alpha$ (in \eqref{Eq:max-stable} and in the following, all maxima are meant to be taken componentwise). The process is furthermore regularly varying with corresponding 
spectral tail process 
$(\Theta_t^{\mbox{\scriptsize RS}})_{t \in \mathbb{Z}}$ as defined in Theorem 
\ref{the:RS}.
\end{theorem}
\begin{proof} 
The process defined in \eqref{Eq:max-stable} is stationary, since 
\begin{eqnarray*}(Z_{t+h})_{t \in \Z}&=& \left(\bigvee_{i \in \N} U_{i} 
\frac{\Theta^{(i)}_{t+h+T_i}}{\|(\Theta_l^{(i)})_{l \in \mathbb{Z}}\|_\alpha}  
\right)_{t \in \mathbb{Z}}\\
&\overset{d}{=}& \left(\bigvee_{i \in \N} U_{i} 
\frac{\Theta^{(i)}_{t+T_i}}{\|(\Theta_l^{(i)})_{l \in \mathbb{Z}}\|_\alpha}  
\right)_{t \in \mathbb{Z}}=(Z_t)_{t \in \Z}
\end{eqnarray*}
for all $h \in \Z$, because the point processes $(T_i)_{i \in \N}$ and 
$(T_i+h)_{i \in \N}$ have the same intensity $\lambda$ and are both independent 
of all other random variables. Similarly as in \cite{HaFe06}, Chapter 9, it 
follows that for $s \leq t, x_n=(x_n^1, \ldots, x_n^d) \in [0,\infty)^d, s \leq 
n \leq t$ (and with $[0,x_n]:=[0,x_n^1]\times \ldots \times [0,x_n^d]$)
\begin{eqnarray}
\nonumber && P(Z_s\leq x_s, \ldots, Z_t\leq x_t)\\
\nonumber &=&\exp\left(-\int \int \int \mathds{1}_{\{([0,x_s] \times \cdots 
\times [0,x_t])^c\}}\left(u\frac{(\theta_{n+z})_{s \leq n \leq 
t}}{\|\theta\|_\alpha}\right)\nu_{\alpha}(du) \lambda(dz)P^{(\Theta_t)_{t \in 
\Z}}(d\theta)\right) \\
\nonumber &=& \exp\left(-\int \|\theta\|_\alpha^{-\alpha} \sum_{z \in 
\mathbb{Z}} \left(\min_{\substack{1\leq i \leq d \\ s\leq n \leq 
t}}\frac{x_n^i}{\theta_{n+z}^i}\right)^{-\alpha}P^{(\Theta_t)_{t \in 
\Z}}(d\theta)\right)\\
\label{Eq:df}&=&  \exp\left(- \sum_{z \in \mathbb{Z}} E\left( \|(\Theta_t)_{t 
\in \Z}\|_\alpha^{-\alpha}\left(\max_{\substack{1\leq i \leq d \\ s\leq n \leq 
t}}\frac{\Theta_{n+z}^i}{x_n^i}\right)^{\alpha}\right)\right),
\end{eqnarray}
where $\nu_\alpha((x,\infty])=x^{-\alpha}, x>0$. This defines a proper distribution function and the process is thus almost 
surely finite. Furthermore, one easily sees that Equation \eqref{eq:defms} 
holds and that the process is therefore max-stable with index $\alpha$. 
To show the regular variation of the time series, let $s\leq 0 \leq t$ and $x_s 
\ldots, x_t$ with $(x_s, \ldots, x_t)\neq0$ be as above such that for $y>0$
\begin{eqnarray*}
&& y^{\alpha}P((Z_s/y, \ldots, Z_t/y) \in ([0,x_s]\times \cdots \times 
[0,x_t])^c)\\
&=& y^{\alpha}\left(1-\exp\left(-\int \int \int \mathds{1}_{\{([0,x_s] \times 
\cdots \times [0,x_t])^c\}}\left(u\frac{(\theta_{n+z})_{s \leq n \leq t}}{y 
\|\theta\|_\alpha}\right)\nu_{\alpha}(du) \lambda(dz)P^{(\Theta_t)_{t \in 
\Z}}(d\theta)\right)\right) \\
&=& y^{\alpha}\left(1-\exp\left(-\frac{1}{y^{\alpha}}\int \int \int 
\mathds{1}_{\{([0,x_s] \times \cdots \times 
[0,x_t])^c\}}\left(v\frac{(\theta_{n+z})_{s \leq n \leq 
t}}{\|\theta\|_\alpha}\right)\nu_{\alpha}(dv) \lambda(dz)P^{(\Theta_t)_{t \in 
\Z}}(d\theta)\right)\right) \\
&\to& \int \int \int \mathds{1}_{\{([0,x_s] \times \cdots \times 
[0,x_t])^c\}}\left(v\frac{(\theta_{n+z})_{s \leq n \leq 
t}}{\|\theta\|_\alpha}\right)\nu_{\alpha}(dv) \lambda(dz)P^{(\Theta_t)_{t \in 
\Z}}(d\theta), \;\;\; y \to \infty.
\end{eqnarray*}
Therefore, $(Z_s, \ldots, Z_t)$ is regularly varying with limit measure 
\begin{eqnarray}\nonumber \mu(A)&=& \int \int \int 
\mathds{1}_{A}\left(v\frac{(\theta_{n+z})_{s \leq n \leq 
t}}{\|\theta\|_\alpha}\right)\nu_{\alpha}(dv) \lambda(dz)P^{(\Theta_t)_{t \in 
\Z}}(d\theta)\\
\label{Eq:limitX}&=& \sum_{z \in \Z}\int \int_0^\infty  
\mathds{1}_{A}\left(v\frac{(\theta_{n+z})_{s \leq n \leq 
t}}{\|\theta\|_\alpha}\right)\nu_{\alpha}(dv) P^{(\Theta_t)_{t \in 
\Z}}(d\theta) \end{eqnarray}
for Borel sets $A\subset ([0,\infty)^d)^{t-s+1}$ bounded away from $0$. Thus, 
for a Borel set $A \subset ([0,\infty)^d)^{t-s+1}$ and $B=([0,\infty)^d)^{-s} 
\times \{(x_1, \ldots, x_d) \in [0,\infty)^d: \|(x_1, \ldots, x_d)\|>1\} \times 
([0,\infty)^d)^{t}$ with $\mu(\partial(A \cap B))=0$, we have
\begin{eqnarray*}
&& \lim_{x \to \infty} P\left(\left(\frac{Z_s}{x}, \ldots, \frac{Z_t}{x} 
\right) \in A \big| \|Z_0\|>x \right) \\
&=& \lim_{x \to \infty} \frac{P\left(\left(\frac{Z_s}{x}, \ldots, 
\frac{Z_t}{x}\right)  \in A \cap B\right) }{P\left(\left(\frac{Z_s}{x}, \ldots, 
\frac{Z_t}{x}\right)  \in B\right) } \\
&=& \frac{\mu(A \cap B)}{\mu(B)}.
\end{eqnarray*}
Using \eqref{Eq:limitX} and substituting $u_z=v\|\theta_z\|/\|\theta\|_\alpha, 
z \in \Z$, we get
\begin{eqnarray*}
\mu(A \cap B)&=& \sum_{z \in \Z}\int \int_0^\infty \mathds{1}_{A \cap 
B}\left(v\frac{(\theta_{n+z})_{s \leq n \leq 
t}}{\|\theta\|_\alpha}\right)\nu_{\alpha}(dv)P^{(\Theta_t)_{t \in 
\Z}}(d\theta)\\
&=& \sum_{z \in \Z}\int \int_0^\infty \mathds{1}_{A \cap 
B}\left(u_z\frac{(\theta_{n+z})_{s \leq n \leq 
t}}{\|\theta_z\|}\right)\frac{\|\theta_z\|^\alpha}{\|\theta\|_\alpha^\alpha}\nu_
{\alpha}(du_z)P^{(\Theta_t)_{t \in \Z}}(d\theta)\\
&=& \int  \sum_{z \in \Z} 
\int_1^\infty\mathds{1}_{A}\left(u_z\frac{(\theta_{n+z})_{s \leq n \leq 
t}}{\|\theta_z\|}\right)\nu_{\alpha}(du_z) 
\frac{\|\theta_z\|^\alpha}{\|\theta\|_\alpha^\alpha}P^{(\Theta_t)_{t \in 
\Z}}(d\theta),
\end{eqnarray*}
where we used for the substitution that, due to the definition of $B$, the 
integrand in the first line is 0 if $\|\theta_z\|=0$ and we used in the last 
equation that $u_z(\theta_{n+z})_{s \leq n \leq t}/\|\theta_z\| \in B$ if and 
only if $u_z \geq 1$. But the last expression is equal to
$$ P((Y \cdot \Theta_{s}^{\mbox{\scriptsize{RS}}}, \ldots, Y \cdot 
\Theta_t^{\mbox{\scriptsize{RS}}}) \in A)$$
for a random variable $Y$ with Pareto$(\alpha)$ distribution independent of 
$(\Theta_s^{\mbox{\scriptsize{RS}}}, \ldots, 
\Theta_t^{\mbox{\scriptsize{RS}}})$. Analogously, one can show that $\mu(B)=1$ 
and thus, by Remark \ref{rem:tailproc}, the process $(Z_t)_{t \in \Z}$ has 
spectral tail process $(\Theta_t^{\mbox{\scriptsize{RS}}})_{t \in \Z}$, which 
completes the proof.
\end{proof}
The following corollary addresses the same question as Theorem 5.1 in \cite{PlSo17} but gives an alternative construction of the underlying max-stable process.
\begin{corollary}\label{Cor:main}
If $(\Theta_t)_{t \in \Z}$ satisfies both Property (SC) and Property (TCF), 
then the spectral tail process of $(Z_t)_{t \in \Z}$ as defined in 
\eqref{Eq:max-stable} is given by $(\Theta_t)_{t \in \Z}$.
\end{corollary}
\begin{proof}
 This follows immediately from Theorem \ref{the:main} in connection with 
Theorem \ref{the:RS}.
\end{proof}

\begin{remark}\label{rem:nonpos}
The construction of a max-stable process as in Theorem \ref{the:main} only 
works if all components of the spectral tail process are non-negative. To 
overcome this and to describe the general dependence structure between 
observations that are extremely large or small (i.e.\ smaller than $-c$ for 
$c\to \infty$) we can look instead at the $2d$-dimensional, non-negative 
process 
$(\Theta_t^{\pm})_{t \in \Z}:=((\Theta_t^1)_+,(\Theta_t^1)_-, \ldots,$ 
$(\Theta_t^d)_+,(\Theta_t^d)_-)_{t \in \Z}$ with $x_+=\max(x,0), 
x_-=\max(-x,0)$. 
\end{remark}
Next, we show that the max-stable process constructed in Theorem 
\ref{the:main} is actually the ``maximum attractor" (i.e.\ the limiting 
distribution for maxima) for the process underlying $(\Theta_t)_{t \in \Z}$. 
All maxima are meant to be taken componentwise.
\begin{proposition}\label{Prop:max-stable}
 Let $(X_t)_{t \in \Z}$ with values in $\R^d$ be a stationary regularly varying 
time series with index $\alpha>0$ and spectral tail process $(\Theta_t)_{t \in 
\mathbb{Z}}$ with values in $[0,\infty)^d$ that satisfies Property (SC). Then, 
for each $t \in \N_0,$ there exists a sequence $b_n>0, n \in \N,$ of 
normalizing constants such that 
 $$ \bigvee_{i=1}^n \frac{(X_0, \ldots, X_t)^i}{b_n} \overset{w}{\Rightarrow} 
(Z_0, \ldots, Z_t), \;\;\; n \to \infty,$$
 where $(X_0, \ldots, X_t)^i, i \in \N,$ are i.i.d.\ copies of $(X_0, \ldots, 
X_t)$ and $(Z_0, \ldots, Z_t)$ has the distribution as defined in Theorem 
\ref{the:main}. 
\end{proposition}
\begin{proof}
Note first that since $(X_t)_{t \in \Z}$ is stationary, the process 
$(\Theta_t)_{t \in \Z}$ satisfies Property (TCF) and thus all assumptions of 
Theorem \ref{the:main} and Corollary \ref{Cor:main} are satisfied. The 
stationary time series $(X_t)_{t \in \Z}$ is multivariate regularly varying if 
and only if all $(X_0, \ldots, X_t), t \in \N_0,$ are multivariate regularly 
varying. By Proposition 7.1 in \cite{Re07} this is equivalent to the existence 
of a normalizing sequence $\tilde{b}_n>0, n \in \N,$ such that
$$ \bigvee_{i=1}^n \frac{(X_0, \ldots, X_t)^i}{\tilde{b}_n} 
\overset{w}{\Rightarrow} (U_0, \ldots, U_t), \;\;\; n \to \infty,$$
for some random vector $(U_0, \ldots, U_t)$ with a non-degenerate distribution. 
This distribution is (up to a scaling factor which only depends on 
$\tilde{b}_n, n \in \N)$ determined by the limit law of regular variation of 
$(X_0, \ldots, X_t)$ and this  is in turn (up to a constant) determined by the 
spectral tail process $(\Theta_t)_{t \in \Z}$ and the index $\alpha$ of regular 
variation, see the proof of Theorem 2.1 in \cite{BaSe09}. Therefore, since the 
processes $(Z_t)_{t \in \Z}$ from Theorem \ref{the:main} and $(X_t)_{t \in \Z}$ 
have the same spectral tail process and index of regular variation, there exist 
$\tilde{b}_n'>0, n \in \N,$ such that
$$ \bigvee_{i=1}^n \frac{(Z_0, \ldots, Z_t)^i}{\tilde{b}_n'} 
\overset{w}{\Rightarrow} (U_0, \ldots, U_t), \;\;\; n \to \infty,$$
where $(Z_0, \ldots, Z_t)^i, i \in \N,$ are i.i.d.\ copies of $(Z_0, \ldots, 
Z_t)$.
But we also know from Definition \ref{def:maxstable} and Theorem \ref{the:main} 
 that
$$ \bigvee_{i=1}^n \frac{(Z_0, \ldots, Z_t)^i}{n^{1/\alpha}} \overset{d}{=} 
(Z_0, \ldots, Z_t), $$
which implies by the convergence to types theorem that 
$$ (U_0, \ldots, U_t) \overset{d}{=} c\cdot (Z_0, \ldots, Z_t) $$
for some $c>0$. Setting $b_n=c\tilde{b}_n$ then leads to the result.
\end{proof}

Proposition \ref{Prop:max-stable} and Theorem \ref{the:main} combined with 
Corollary \ref{Cor:main} provide a view on the extremal behavior of a random 
process that is somehow complementary to the description given by the spectral 
tail process. The latter one only describes the behavior given that we have 
seen an extremal event \emph{at a specific time}. The process constructed in 
Theorem \ref{the:main} shares the extremal behavior of a process with given 
spectral tail process in the sense of Proposition \ref{Prop:max-stable} and 
thus 
gives us an impression about when the extremal events will happen over time. 
``Extremal episodes'' of the process in \eqref{Eq:max-stable} will typically 
occur for an extremely large value of $U_i$ and the corresponding 
``epicenters'' 
$T_i$ are uniformly distributed over time. Due to Property (SC), the influence 
of a large value of $U_i$ will only be visible in a certain neighborhood around 
the corresponding time $T_i$ and, roughly speaking, extremal events will become 
more and more independent if the elapsed time between them grows. 

From a statistical point of view, it is important to connect the extremal 
behavior that we observe in the ``extremal episodes" of a time series to the 
extremal behavior of the stationary distribution, i.e. the distribution at a 
fixed point in time (or maybe even the joint distribution at specific lags). 
The extremal behavior of the stationary distribution can be described by 
$\lim_{x \to \infty}\mathcal{L}(X_0/\|X_0\| \mid 
\|X_0\|>x)=\mathcal{L}(\Theta_0)$. Surely, the extremal realizations of a time 
series are the ones to look at for this task, but if we have an extremal 
cluster 
of events we look at the process at \emph{random} points in time. One may then 
ask which of those extremal observations from an observed cluster best 
represents a ``typical'' extremal observation, i.e.\ an extremal observation 
from the stationary distribution. This question is even more important if one 
wants to decrease dependence between used observations and therefore chooses 
only 
one observation per cluster for inference. This is for example done in the 
declustering approach of the POT method as introduced in \cite{DaSm90}, and has 
by now become a standard tool for the extremal analysis of time series. Here 
one 
chooses (in a univariate setting) the cluster maximum as a representative for 
the whole cluster and treats the resulting observations as extremal outcomes of 
the stationary distribution which are nearly i.i.d. As Proposition 
\ref{Prop:max-stable} shows, the extremal clusters can be seen as realizations 
of $(\Theta_t)_{t \in \Z}$ with random scaling and under an unobservable shift 
in time, $T_i$. So we will usually not observe $(\Theta_t)_{t \in \Z}$ but, 
with 
the notation from Section \ref{Sec:ShiftInv}, only 
$(\Theta_{T^\ast+t}/\|\Theta^\ast\|)_{t \in \Z}$, that is the observed ``pattern'' in form of the self-standardized 
process that has maximum norm 1 and is shifted in a way such that the maximal 
norm is first attained at time zero. Alternatively, the observation could also 
be seen as a realization in the quotient space of double-sided sequences with 
respect to the shift operator, cf.\ the space $\tilde{l}_0$ in \cite{BaPlSo16}. 
The conditional distribution of $(\Theta_t)_{t \in \Z}$, given the shifted and 
rescaled observation as just described, is found in the next proposition.
\begin{proposition}\label{prop:cond}
Let $(\Theta_t)_{t \in \Z}$ be a time series which satisfies Property (TCF) and 
Property (SC). Then, with the notation from \eqref{Eq:Thetastar}, 
\eqref{Eq:Tstar},
\begin{equation}\label{Eq:shiftcond} \mathcal{L}\left((\Theta_t)_{t \in 
\Z}\middle|\left(\frac{\Theta_{T^\ast+t}}{\|\Theta^\ast\|}\right)_{t \in 
\Z}\right)=\sum_{k\in \Z}\frac{\|\Theta_{T^\ast+k}\|^\alpha}{\|(\Theta_{t})_{t 
\in \Z}\|_\alpha^\alpha}\, 
\delta_{\left(\frac{\Theta_{T^\ast+k+t}}{\|\Theta_{T^\ast+k}\|}\right)_{t \in 
\Z}}, 
\end{equation}
where $\delta_x$ denotes the Dirac measure in $x \in (\R^d)^{\Z}$.
\end{proposition}
\begin{proof}
Note first that the random probability measure on the r.h.s.\ in 
\eqref{Eq:shiftcond}, applied to some $A\subset \mathcal{B}((\R^d)^\Z)$, is 
equal to
\begin{eqnarray*}
&& \sum_{k\in \Z}\frac{\|\Theta_{T^\ast+k}\|^\alpha}{\|(\Theta_{t})_{t \in 
\Z}\|_\alpha^\alpha}\, 
\mathds{1}_A\left(\left(\frac{\Theta_{T^\ast+k+t}}{\|\Theta_{T^\ast+k}\|}
\right)_{t \in \Z}\right)\\
&=& \sum_{k\in 
\Z}\frac{\|\Theta_{T^\ast+k}/\|\Theta^\ast\|\|^\alpha}{\|(\Theta_{T^\ast+t})_{t 
\in \Z}/\|\Theta^\ast\|\|_\alpha^\alpha}\, 
\mathds{1}_A\left(\left(\frac{\Theta_{T^\ast+k+t}/\|\Theta^\ast\|}{\|\Theta_{
T^\ast+k}/\|\Theta^\ast\|\|}\right)_{t \in \Z}\right)
\end{eqnarray*}
and is thus a measurable function of the conditioning expression. Let now $A, B 
\subset \mathcal{B}((\R^d)^\Z)$. Then, from Theorem \ref{the:RS}, 
\begin{eqnarray*}
&& P\left((\Theta_t)_{t \in \Z} \in A, 
\left(\frac{\Theta_{T^\ast+t}}{\|\Theta^\ast\|}\right)_{t \in \Z} \in B\right)\\
&=&E\left(\sum_{l \in \Z} \frac{\|\Theta_l\|^\alpha}{\|(\Theta_t)_{t \in 
\Z}\|_\alpha^\alpha}\mathds{1}_A\left(\left(\frac{\Theta_{l+t}}{\|\Theta_l\|}
\right)_{t \in 
\Z}\right)\mathds{1}_B\left(\left(\frac{\Theta_{T^\ast((\Theta_{l+t})_{t \in 
\Z})+l+t}}{\|\Theta^\ast\|((\Theta_{l+t})_{t \in \Z})}\right)_{t \in 
\Z}\right)\right)\\
&=& E\left(\sum_{l \in \Z} \frac{\|\Theta_l\|^\alpha}{\|(\Theta_t)_{t \in 
\Z}\|_\alpha^\alpha}\mathds{1}_A\left(\left(\frac{\Theta_{l+t}}{\|\Theta_l\|}
\right)_{t \in 
\Z}\right)\mathds{1}_B\left(\left(\frac{\Theta_{T^\ast+t}}{\|\Theta^\ast\|}
\right)_{t \in \Z}\right)\right)\\
&=& E\left(\sum_{k \in \Z} \frac{\|\Theta_{T^\ast+k}\|^\alpha}{\|(\Theta_t)_{t 
\in 
\Z}\|_\alpha^\alpha}\mathds{1}_A\left(\left(\frac{\Theta_{T^\ast+k+t}}{\|\Theta_
{T^\ast+k}\|}\right)_{t \in 
\Z}\right)\mathds{1}_B\left(\left(\frac{\Theta_{T^\ast+t}}{\|\Theta^\ast\|}
\right)_{t \in \Z}\right)\right),
\end{eqnarray*}
which finishes the proof.
\end{proof}
A way of interpreting Proposition \ref{prop:cond} is that given an observed extremal ``pattern'' $(\Theta_{T^\ast+t}/\|\Theta^\ast\|)_{t \in \Z}$ from the realization of a time series, the resulting conditional distribution of the underlying spectral tail process is a random shift of this pattern (where the probability that an observation is set to be at time 0 is proportional to the norm of that observation to the power of $\alpha$), and scaled in a way to ensure that $\|\Theta_0\|=1$. Furthermore, a closer look at \eqref{Eq:shiftcond} shows that given an observation 
of $(\Theta_{T^\ast+t}/\|\Theta^\ast\|)_{t \in \Z}$ (and if the maximum norm is 
attained only once) this very sequence itself is also the most likely of all 
possible underlying sequences $(\Theta_t)_{t \in \Z}$, because the weight in 
\eqref{Eq:shiftcond} is largest for $k=0$. Therefore, $\Theta_{T^\ast}/\|\Theta^\ast\|$ can be seen as the best approximation of $\Theta_0$ from the observed sequence, which makes the representational choice of the observation with maximal norm from a cluster as in the POT method reasonable.

\section{Construction of a max-stable process with given spectral tail process 
in the general case}\label{Sec:general}
So far, we have focussed on processes $(\Theta_t)_{t \in \Z}$ which satisfy 
both Property (TCF) and Property (SC). Corollary \ref{Cor:main} shows that this 
ensures the existence of a max-stable underlying process which realizes 
$(\Theta_t)_{t \in \Z}$ as a spectral tail process. In fact, this max-stable 
process is of a specific form which is called a \emph{mixed moving maxima (M3) 
process}, cf.\ for example Definition 7 in \cite{DoKa16} for the univariate 
case where $\alpha=1$. \cite{DoKa16} show that a (univariate) max-stable 
process 
has a representation as a M3 process of the above form if and only if the 
process is (purely) dissipative, cf.\ Theorem 8 in \cite{DoKa16} and Theorem 
5.4 
in \cite{WaSt10}, where also the case for general $\alpha$ is covered 
explicitly. Intuitively, this case implies that the impact of an extremal event 
at time 0 may in principle last forever (since $\|\Theta_t\|>0$ for all $t 
\in \Z$ is possible) but that it diminishes over time since $\|\Theta_t\| \to 
0$ 
as $|t|\to \infty$ almost surely. 

We shall now look at the case where $(\Theta_t)_{t \in \Z}$ satisfies Property 
(TCF) but not necessarily Property (SC) for a given value of $\alpha>0$. If 
Property (SC) is not satisfied, then this implies by Lemma \ref{equivalence} 
that $P(T^\ast \notin \Z)>0$ and, roughly speaking, this corresponds to the 
case where an extremal event at time 0 will actually ``return'' infinitely 
often, for example in a periodic manner, like the following example shows.
\begin{example}[Simple periodic spectral tail process]\label{ex:simpex}
Let $d=1$ and 
$$ \Theta_t=\begin{cases}
            1 & \mbox{ if } t \in 2\Z, \\
            0 & \mbox{ if } t \in 2\Z+1.
            \end{cases}$$
One easily checks that the process $(\Theta_t)_{t \in \Z}$ satisfies Property 
(TCF) for any $\alpha>0$ but there exists no $\alpha>0$ such that 
$(\Theta_t)_{t \in \Z}$ also satisfies Property (SC). 
Define now
$$ (X_t)_{t \in \mathbb{Z}}=\left(\bigvee_{j=0}^1 \bigvee_{i \in \N} U_{i}^{(j)} \Theta^{(j,i)}_{t+j} \right)_{t \in \mathbb{Z}},$$
where $(U_i^{(j)}, (\Theta_t^{(j,i)})_{t \in \mathbb{Z}})_{i \in \mathbb{N}}$ for $j=0,1$ is an enumeration of points from a Poisson point process on $(0,\infty) \times  
([0,\infty))^{\Z}$ with intensity $\alpha u^{-\alpha-1} du \otimes 
P^{(\Theta_t)_{t \in \Z}}(d\theta)$ and let those two Poisson point processes be 
independent for $j=0$ and $j=1$. It is easily seen that $(X_t)_{t \in \Z}$ is a stationary max-stable process and that
$$ (X_t)_{t \in \mathbb{Z}}\overset{d}{=} (Z_0\Theta_t\mathds{1}_{2 \mathbb{Z}}(t)+Z_1\Theta_{t+1}\mathds{1}_{2 \mathbb{Z}+1}(t))_{t \in \mathbb{Z}}=(Z_0\mathds{1}_{2 \mathbb{Z}}(t)+Z_1\mathds{1}_{2 \mathbb{Z}+1}(t))_{t \in \mathbb{Z}}$$
for $Z_0, Z_1$ being independent and Fr\'{e}chet($\alpha$)-distributed. Therefore, the 
resulting spectral tail process of $(X_t)_{t \in \Z}$ is given by 
$(\Theta_t)_{t 
\in \Z}$. 
\end{example}
The following theorem shows how the construction principle from the above 
example can be generalized in order to construct a corresponding stationary 
max-stable process for a general process $(\Theta_t)_{t \in \Z}$ which 
satisfies Property (TCF). We restrict ourselves again to the non-negative case, 
but note that an analogue of Remark \ref{rem:nonpos} holds in this case as 
well. See \cite{DoHaSo17}, Theorem 2.9, for a different approach that shows the existence of an underlying process based on the concept of the tail measure.  
 
\begin{theorem}\label{the:general}
Let $(\Theta_t)_{t \in \Z}$ with values in $[0,\infty)^d$ be a stochastic 
process which satisfies Property (TCF). Furthermore, for $j \in \Z$ let 
$(U_i^{(j)}, (\Theta_t^{(j,i)})_{t \in \mathbb{Z}})_{i \in \mathbb{N}}$ be an 
enumeration of points from a Poisson point process on $(0,\infty) \times  
([0,\infty)^d)^{\Z}$ with intensity $\alpha u^{-\alpha-1} du \otimes 
P^{(\Theta_t)_{t \in \Z}}(d\theta)$ and let those Poisson point processes be 
independent for different values of $j$. Define for $k>0$ the sets
$$ Q_{k}=\{(\theta_t)_{t \in \mathbb{Z}}: \theta_0\neq 0, 
\theta_1=\theta_2=\ldots=\theta_{2k}=0\},$$
$$ \; Q_{-k}=\{(\theta_t)_{t \in 
\mathbb{Z}}:\theta_{-2k+1}=\theta_{-2k+2}=\ldots=\theta_{-1}=0, \theta_0\neq 
0\}$$
and set $Q_0=\{(\theta_t)_{t \in \mathbb{Z}}: \theta_0 \neq 0\}$.
Then the stochastic process
\begin{equation}\label{Eq:general} (Z_t)_{t \in \Z}=\left(\bigvee_{j \in 
\mathbb{Z}} \bigvee_{i \in \N} U_{i}^{(j)} \mathds{1}_{\{(\Theta^{(j,i)}_t)_{t 
\in \mathbb{Z}} \in Q_{j}\}}\Theta^{(j,i)}_{t+j} \right)_{t \in \Z} 
\end{equation}
is an almost surely finite, stationary and max-stable process with index 
$\alpha$. The process is furthermore regularly varying with corresponding 
spectral tail process 
$(\Theta_t)_{t \in \Z}$.
\end{theorem}
\begin{proof} 
Note first that for all $n, k \in 
\mathbb{N}_0$ and $x=(x_0, \ldots, x_n) \in ([0,\infty)^d)^{n+1}$ we have, analogously to the proof of Theorem \ref{the:main}, that
\begin{eqnarray}&&\nonumber P(Z_{-k}\leq x_0, \ldots, Z_{n-k} \leq x_n)\\
&=&\nonumber \exp\left(-\sum_{j \in \mathbb{Z}}\int \int 
\mathds{1}_{[0,x]^c}\left(y(\theta_{t+j})_{-k \leq t \leq 
n-k}\mathds{1}_{\{(\theta_t)_{t \in \mathbb{Z}} \in 
Q_{j}\}}\right)\nu_\alpha(dy)P^{(\Theta)_{t \in \Z}}(d\theta)\right)\\
&=&\label{Eq:generalexp} \exp\left(-\sum_{j \in \mathbb{Z}}\int \int 
\mathds{1}_{[0,x]^c}\left(y(\theta_{t+j-k})_{0 \leq t \leq 
n}\mathds{1}_{\{(\theta_t)_{t \in \mathbb{Z}} \in 
Q_{j}\}}\right)\nu_\alpha(dy)P^{(\Theta)_{t \in \Z}}(d\theta)\right),
\end{eqnarray} 
where again $\nu_{\alpha}((x,\infty])=x^{-\alpha}, x>0$. Due to the definition of the sets $Q_j$, only finitely many summands in the exponent in \eqref{Eq:generalexp} are different from 0. Therefore, if all components of $x_0, \ldots, x_n$ go to infinity, the expression \eqref{Eq:generalexp} converges to 1, 
which shows that $(Z_t)_{t \in \Z}$ is an almost surely finite process.

The stationarity now follows from \eqref{Eq:generalexp} if we can show that
\begin{eqnarray}
 \nonumber && \sum_{j \in \mathbb{Z}}\int \int 
\mathds{1}_{A}\left(y(\theta_{t+j-k})_{0 \leq t \leq 
n}\mathds{1}_{\{(\theta_t)_{t \in \mathbb{Z}} \in 
Q_{j}\}}\right)\nu_\alpha(dy)P^{(\Theta)_{t \in \Z}}(d\theta) \\
 \label{Eq:equationintensity} &=& \sum_{j \in \mathbb{Z}}\int \int 
\mathds{1}_{A}\left(y(\theta_{t+j})_{0 \leq t \leq 
n}\mathds{1}_{\{(\theta_t)_{t \in \mathbb{Z}} \in 
Q_{j}\}}\right)\nu_\alpha(dy)P^{(\Theta)_{t \in \Z}}(d\theta).
\end{eqnarray}
for all $k, n \in \mathbb{N}_0$ and all Borel sets $A \subset ([0,\infty)^d)^{n+1}$ which 
are bounded away from 0.
Note that \eqref{Eq:equationintensity} follows for all $A$ bounded away from 0 
as soon as we can show it for all $A$ such that $A \cap \{(x_0, \ldots, x_n): 
x_0=0\}=\emptyset$, which can be seen as follows: For $n=0$ all sets $A$ which are 
bounded away from 0 already satisfy $A \cap \{(x_0): x_0=0\}=\emptyset$. For 
$n\geq 1$ any $A$ bounded away from 0 is the disjoint union of the sets
$$ A\cap \{(x_0, \ldots, x_n):x_0\neq 0\})$$
and
\begin{eqnarray*} A\cap \{(x_0, \ldots, x_n):x_0=0\})&=&\{(x_0, x_1, \ldots, 
x_n):(x_1, \ldots, x_n) \in A'\})\\
&& \setminus \{(x_0, x_1, \ldots, x_n):(x_1, \ldots, x_n) \in A', x_0\neq 0\})
\end{eqnarray*}
for a set $A'\subset ([0,\infty)^d)^{n}$ bounded away from 0. By induction, the 
statement then follows.

So assume in the following that $A \cap \{(x_0, \ldots, x_n): 
x_0=0\}=\emptyset$, such that
\begin{eqnarray*}
 && \sum_{j \in \mathbb{Z}}\int \int \mathds{1}_{A}\left(y(\theta_{t+j-k})_{0 
\leq t \leq n}\mathds{1}_{\{(\theta_t)_{t \in \mathbb{Z}} \in 
Q_{j}\}}\right)\nu_\alpha(dy)P^{(\Theta)_{t \in \Z}}(d\theta) \\
 &=& \sum_{j \in \mathbb{Z}}\int \int \mathds{1}_{A}\left(y(\theta_{t+j-k})_{0 
\leq t \leq n}\mathds{1}_{\{(\theta_t)_{t \in \mathbb{Z}} \in 
Q_{j}\}}\right)\mathds{1}_{\{\theta_{j-k}\neq 0\}}\nu_\alpha(dy)P^{(\Theta)_{t 
\in \Z}}(d\theta)\\
 &=& \sum_{j \in \mathbb{Z}}E\left( \int 
\mathds{1}_{A}\left(y(\Theta_{t+j-k})_{0 \leq t \leq 
n}\mathds{1}_{\{(\Theta_t)_{t \in \mathbb{Z}} \in 
Q_{j}\}}\right)\mathds{1}_{\{\Theta_{j-k}\neq 0\}}\right)\nu_\alpha(dy) \\
 &=& \sum_{j \in \mathbb{Z}}E\left( \int 
\mathds{1}_{A}\left(y\left(\frac{\Theta_{t}}{\|\Theta_{k-j}\|}\right)_{0 \leq t 
\leq 
n}\mathds{1}_{\left\{\left(\frac{\Theta_{t+k-j}}{\|\Theta_{k-j}\|}\right)_{t 
\in \mathbb{Z}} \in Q_{j}\right\}}\right)\mathds{1}_{\{\Theta_0\neq 
0\}}\|\Theta_{k-j}\|^\alpha\right)\nu_\alpha(dy) \\
 &=& \sum_{j \in \mathbb{Z}}\int \int 
\mathds{1}_{A}\left(y\left(\frac{\theta_{t}}{\|\theta_{k-j}\|}\right)_{0 \leq t 
\leq 
n}\mathds{1}_{\left\{\left(\frac{\theta_{t+k-j}}{\|\theta_{k-j}\|}\right)_{t 
\in \mathbb{Z}} \in 
Q_{j}\right\}}\right)\|\theta_{k-j}\|^\alpha\nu_\alpha(dy)P^{(\Theta)_{t \in 
\Z}}(d\theta)
\end{eqnarray*}
where we used that $(\Theta_t)_{t \in \mathbb{Z}}$ satisfies property (TCF). 
Since $(\theta_t)_{t \in \mathbb{Z}} \in Q_j$ if and only if $(c\theta_t)_{t 
\in \mathbb{Z}} \in Q_j$ for any $c>0$, and by substituting $u$ for 
$y/\|\theta_{k-j}\|$ in each summand, this equals
\begin{eqnarray*}
 && \sum_{j \in \mathbb{Z}}\int \int \mathds{1}_{A}\left(u(\theta_{t})_{0 \leq 
t \leq n}\mathds{1}_{\left\{(\theta_{t+k-j})_{t \in \mathbb{Z}} \in 
Q_{j}\right\}}\right)\mathds{1}_{\{\theta_{k-j}\neq 
0\}}\nu_\alpha(du)P^{(\Theta)_{t \in \Z}}(d\theta) \\
 &=&  \sum_{j \in \mathbb{Z}}\int \int \mathds{1}_{A}(u(\theta_{t})_{0 \leq t 
\leq n})\mathds{1}_{\{\theta_{k-j}\neq 0, (\theta_{t+k-j})_{t \in \mathbb{Z}} 
\in Q_{j}\}}\nu_\alpha(du)P^{(\Theta)_{t \in \Z}}(d\theta), 
\end{eqnarray*}
since $A$ is bounded away from 0. Now, the events in the indicator functions 
above are equivalent to
\begin{align*}
&\theta_k \neq 0, & j&=0,\\
&\theta_{k+1} \neq 0, \theta_k=0, & j&=-1, \\
&\theta_{k-1} \neq 0, \theta_k=\theta_{k+1}=0, & j&=1, \\
&\theta_{k+2} \neq 0, \theta_{k-1}=\theta_k=\theta_{k+1}=0, & j&=-2 \\
&\vdots & 
\end{align*}
and so the sum of the indicator functions of these disjoint events is equal to 
$\mathds{1}_{\{(\theta_t)_{t \in \Z} \neq 0\}}$, which equals 1 almost surely 
due to $P(\|\Theta_0\|=1)=1$. Therefore, 
\begin{eqnarray*}
&& \sum_{j \in \mathbb{Z}}\int \int \mathds{1}_{A}\left(y(\theta_{t+j-k})_{0 
\leq t \leq n}\mathds{1}_{\{(\theta_t)_{t \in \mathbb{Z}} \in 
Q_{j}\}}\right)\nu_\alpha(dy)P^{(\Theta)_{t \in \Z}}(d\theta)\\
&=& \int \int \mathds{1}_{A}(u(\theta_{t})_{0 \leq t \leq 
n})\nu_\alpha(du)P^{(\Theta)_{t \in \Z}}(d\theta)\\
&=& \sum_{j \in \mathbb{Z}}\int \int \mathds{1}_{A}\left(u(\theta_{t+j})_{0 
\leq t \leq n}\mathds{1}_{\{(\theta_t)_{t \in \mathbb{Z}} \in 
Q_{j}\}}\right)\nu_\alpha(du)P^{(\Theta)_{t \in \Z}}(d\theta),
\end{eqnarray*}
where in the last equality we used that $(\theta_t)_{t \in \Z} \in Q_j$ implies 
that $\theta_j=0$ for all $j \neq 0$ and therefore $u(\theta_{t+j})_{0 
	\leq t \leq n}\mathds{1}_{\{(\theta_t)_{t \in \mathbb{Z}} \in 
	Q_{j}\}} \notin A$ for $j\neq 0$. This shows \eqref{Eq:equationintensity} and thereby the 
stationarity of $(Z_t)_{t \in \N}$.  

For any $x=(x_0, \ldots, x_n) \in ([0,\infty)^d)^{n+1}$ and $k \in \N$ we have 
furthermore (by substituting $u$ for $k^{1/\alpha}y$) that
\begin{eqnarray*}
&& P(Z_0 \leq k^{-1/\alpha}x_0, \ldots, Z_n \leq k^{-1/\alpha}x_n) \\
&=& \exp\left(-\sum_{j \in \Z}\int \int 
\mathds{1}_{[0,k^{-1/\alpha}x]^c}\left(y(\theta_{t+j})_{0 \leq t \leq 
n}\mathds{1}_{\{(\theta_t)_{t \in \mathbb{Z}} \in 
Q_{j}\}}\right)\nu_\alpha(dy)P^{(\Theta)_{t \in \Z}}(d\theta)\right)\\
&=& \exp\left(-k\sum_{j \in \Z}\int \int 
\mathds{1}_{[0,x]^c}\left(u(\theta_{t+j})_{0 \leq t \leq 
n}\mathds{1}_{\{(\theta_t)_{t \in \mathbb{Z}} \in 
Q_{j}\}}\right)\nu_\alpha(du)P^{(\Theta)_{t \in \Z}}(d\theta)\right)\\
&=&(P(Z_0 \leq x_1, \dots, Z_n \leq x_n))^k,
\end{eqnarray*}
so the process satisfies Definition \ref{def:maxstable} and is thereby 
max-stable. In order to show that the process is also regularly varying, choose 
$n \in \N$ and $x=(x_{-n}, \ldots, x_n) \in ([0,\infty)^d)^{2n+1}, x \neq 0$. Then, for $z>0$, and 
again with a substitution we have
\begin{eqnarray*}
 && z^{\alpha}P((Z_{-n}/z, \ldots, Z_n/z) \in [0,x]^c) \\
 &=& z^{\alpha}\left(1-\exp\left(-\sum_{j \in \Z}\int \int 
\mathds{1}_{[0,zx]^c}\left(y(\theta_{t+j})_{-n \leq t \leq 
n}\mathds{1}_{\{(\theta_t)_{t \in \mathbb{Z}} \in 
Q_{j}\}}\right)\nu_\alpha(dy)P^{(\Theta)_{t \in \Z}}(d\theta) \right)\right)\\
 &=& z^{\alpha}\left(1-\exp\left(-z^{-\alpha}\sum_{j \in \Z}\int \int 
\mathds{1}_{[0,x]^c}\left(y(\theta_{t+j})_{-n \leq t \leq 
n}\mathds{1}_{\{(\theta_t)_{t \in \mathbb{Z}} \in 
Q_{j}\}}\right)\nu_\alpha(dy)P^{(\Theta)_{t \in \Z}}(d\theta) \right)\right)\\
 &\to& \sum_{j \in \Z}\int \int \mathds{1}_{[0,x]^c}\left(y(\theta_{t+j})_{-n 
\leq t \leq n}\mathds{1}_{\{(\theta_t)_{t \in \mathbb{Z}} \in 
Q_{j}\}}\right)\nu_\alpha(dy)P^{(\Theta)_{t \in \N}}(d\theta), \;\;\; z \to 
\infty,
\end{eqnarray*}
which shows that $(Z_{-n}, \ldots, Z_n)$ is regularly varying with limit measure
$$ \mu(A)=\sum_{j \in \Z}\int \int \mathds{1}_{A}\left(y(\theta_{t+j})_{-n \leq t 
\leq n}\mathds{1}_{\{(\theta_t)_{t \in \mathbb{Z}} \in 
Q_{j}\}}\right)\nu_\alpha(dy)P^{(\Theta)_{t \in \Z}}(d\theta)$$
for sets $A$ bounded away from $(0, \ldots, 0)$. Note that, again due to the definition of the sets $Q_j, j \neq 0$,
\begin{eqnarray*} &&\mu(A \cap \{(x_{-n}, \ldots, x_n):\|x_0\|>1\})\\
&=& \sum_{j \in \Z} \int \int \mathds{1}_{A \cap \{(x_{-n}, \ldots, 
x_n):\|x_0\|>1\}}\left(y(\theta_{t+j})_{-n \leq t \leq 
n}\mathds{1}_{\{(\theta_t)_{t \in \mathbb{Z}} \in 
Q_{j}\}}\right)\nu_\alpha(dy)P^{(\Theta)_{t \in \Z}}(d\theta)\\
&=& \int \int_1^\infty \mathds{1}_A\left(y(\theta_{t})_{-n \leq t \leq 
n}\right)\nu_\alpha(dy)P^{(\Theta)_{t \in \Z}}(d\theta)
\end{eqnarray*}
and
$$ \mu(\{(x_{-n}, \ldots, x_n):\|x_0\|>1\})=\int \int_1^\infty 
\mathds{1}_{([0,\infty)^d)^{2n+1}}\left(y(\theta_{t})_{-n \leq t \leq n}\right) 
\nu_\alpha(dy)P^{(\Theta)_{t \in \Z}}(d\theta)=1,$$
such that for Borel sets $A$ with $\mu(\partial(A \cap \{(x_{-n}, \ldots, 
x_n):\|x_0\|>1\}))=0$,
\begin{eqnarray*}
 && P((Z_{-n}/z, \ldots, Z_n/z) \in A \mid \|Z_0\|>z) \\
 &\to& \frac{\mu(A \cap \{(x_{-n}, \ldots, x_n):\|x_0\|>1\})}{\mu(\{(x_{-n}, \ldots, 
x_n):\|x_0\|>1\})}, \;\;\; z \to \infty,\\
 &=& \int \int_1^\infty \mathds{1}_A\left(y(\theta_{t})_{-n \leq t \leq 
n}\right)\nu_\alpha(dy)P^{(\Theta)_{t \in \Z}}(d\theta)\\
&=& P((Y\cdot \Theta_{-n}, \ldots, Y \cdot \Theta_n) \in A),
\end{eqnarray*}
for a Pareto$(\alpha)$ distributed random variable $Y$ independent of 
$(\Theta_t)_{t \in \N_0}$. By Remark \ref{rem:tailproc}, this proves the 
statement.
\end{proof}

An analogue to Proposition \ref{Prop:max-stable} holds in this case as well, 
i.e.\ the process $(Z_t)_{t \in \N_0}$ from Theorem \ref{the:general} is the 
maximum attractor of the process which underlies $(\Theta_t)_{t \in \Z}$. The 
proof of this result is completely along the lines of the proof of Proposition 
\ref{Prop:max-stable}. 

\section*{Acknowledgements}
The author thanks two anonymous referees for constructive comments on an earlier version of this manuscript.

\end{document}